\documentclass[a4paper,11pt]{amsart}

\usepackage{verbatim}
\usepackage{amssymb,amsthm,amsmath,enumerate}
\usepackage{amsrefs}

\numberwithin{equation}{section}

\newtheorem{thm}{Theorem}[section]

\newtheorem*{thmA}{Theorem A}
\newtheorem*{corB}{Corollary B}
\newtheorem*{corC}{Corollary C}
\newtheorem*{thmD}{Theorem D}
\newtheorem*{thmE}{Theorem E}
\newtheorem*{mainp}{Main Problem}
\newtheorem{lem}[thm]{Lemma}
\newtheorem{cor}[thm]{Corollary}
\newtheorem{prop}[thm]{Proposition}
\theoremstyle{definition}

\newtheorem{pb}{Problem}[section]

\theoremstyle{remark}
\newtheorem{rem}[thm]{Remark}

\newcommand{\sbar}{\ensuremath{\bar{s}}}

\newcommand{\nbar}{\ensuremath{\bar{n}}}
\newcommand{\mbar}{\ensuremath{\bar{m}}}
\newcommand{\ubar}{\ensuremath{\bar{u}}}

\newcommand{\bn}{\ensuremath{\mathbb N}}
\newcommand{\bm}{\ensuremath{\mathbb M}}

\newcommand{\br}{\ensuremath{\mathbb R}}

\newcommand{\cQ}{\ensuremath{\mathcal Q}}

\newcommand{\supp}{\operatorname{supp}}

\newcommand{\symdif}{\bigtriangleup}

\newcommand{\ie}{\textit{i.e.,}\ }

\newcommand{\xn}{\ensuremath{(x_n)_{n=1}^\infty}}
\newcommand{\yn}{\ensuremath{(y_n)_{n=1}^\infty}}
\newcommand{\zn}{\ensuremath{(z_n)_{n=1}^\infty}}

\newcommand{\ei}{\ensuremath{(e_i)_{i=1}^\infty}}
\newcommand{\ej}{\ensuremath{(e_j)_{j=1}^\infty}}

\newcommand{\co}{\mathrm{c}_0}
\newcommand{\coo}{\mathrm{c}_{00}}

\newcommand{\sumspace}[4]{\ensuremath{\left(\sum_{#1}^{#2}#3\right)_{#4}}}

\newcommand{\fin}[1]{{[#1]}^{<\omega}}

\newcommand{\sizelek}[1]{{[#1]}^{\leq k}}
\newcommand{\sizeeqk}[1]{{[#1]}^{k}}

\newcommand{\infin}[1]{{[#1]}^{\omega}}

\newcommand{\lip}{\ensuremath{\mathrm{Lip}}}
\newcommand{\diam}{\ensuremath{\mathrm{diam}}}

\newcommand{\vep}{\varepsilon}
\newcommand{\ds}{\displaystyle}

\begin{document}

\title[The coarse geometry of Tsirelson's space and applications]{The coarse geometry of Tsirelson's space and applications}

\author{F.~Baudier}
\address{F.~Baudier, Department of Mathematics, Texas A\&M University, College Station, TX 77843, USA}
\email{florent@math.tamu.edu}
\thanks{The first named author was partially supported by the National Science Foundation under Grant Number DMS-1565826.}

\author{G.~Lancien}
\address{G.~Lancien, Laboratoire de Math\'ematiques de Besan\c con,  CNRS UMR-6623, Universit\'e Bourgogne Franche-Comt\'e, 16 route de  Gray, 25030 Besan\c con C\'edex, France}
\email{gilles.lancien@univ-fcomte.fr}
\thanks{The second named author was supported by the French ``Investissements d'Avenir'' program, project ISITE-BFC (contract
 ANR-15-IDEX-03) and as a participant of the ``NSF Workshop in Analysis and Probability'' at Texas A\&M University.}

\author{Th.~Schlumprecht}
\address{Th.~Schlumprecht, Department of Mathematics, Texas A\&M University, College Station, TX 77843-3368, USA, and Faculty of Electrical Engineering,
Czech Technical University in Prague, Zikova 4, 16627, Prague, Czech Republic
}
\email{schlump@math.tamu.edu}
\date{}

\thanks{The third named author was supported by the National Science Foundation under Grant Number DMS-1464713.}

\maketitle

\begin{abstract} The main result of this article is a rigidity result pertaining to the spreading model structure for Banach spaces coarsely embeddable into Tsirelson's original space $T^*$. Every Banach space that is coarsely embeddable into $T^*$ must be reflexive and all its spreading models must be isomorphic to $\co$. Several important consequences follow from our rigidity result. We obtain a coarse version of an influential theorem of Tsirelson: $T^*$ does not coarsely contain $\co$ nor $\ell_p$ for $p\in[1,\infty)$. We show that there is no infinite dimensional Banach space that coarsely embeds into every infinite dimensional Banach space. In particular, we disprove the conjecture that the separable infinite dimensional Hilbert space coarsely embeds into every infinite dimensional Banach space. The rigidity result follows from a new concentration inequality for Lipschitz maps on the infinite Hamming graphs and taking values in $T^*$, and from the embeddability of the infinite Hamming graphs into Banach spaces that admit spreading models not isomorphic to $\co$. Also, a purely metric characterization of finite dimensionality is obtained.
\end{abstract}

\tableofcontents

\section{Introduction}

\subsection{Motivation}
A natural and powerful approach to classify mathematical objects in a given category, is to discover properties that are invariant with respect to the isomorphisms of the category. The notion of topological dimension, which allows us to distinguish between $\br^2$ and $\br^3$ as topological spaces, is a prime example. These invariants can take various forms, and in metric geometry one usually seeks for properties that are stable under embeddings of a certain type, e.g. Lipschitz, uniform, or coarse. In this article we are concerned with a coarse invariant in the form of a concentration inequality for Lipschitz maps defined on infinite graphs with values into certain infinite dimensional Banach spaces.

\medskip

Discovering useful and tractable new concentration inequalities is usually a challenge. In the Lipschitz category, several fundamental inequalities have been discovered, partially motivated by the investigation of the Ribe program and its applications in geometric group theory or theoretical computer science. We refer to the survey of A. Naor \cite{Naor2012} for an enlightening account of this theory, and to the recent work of Naor and Schechtman (\cite{NSPi2016}, \cite{NSFourier2016} and \cite{NaorAnnals2016}) where fundamental and powerful new inequalities were proved.

\medskip

Strong interest in the coarse geometry of infinite dimensional Banach spaces arose in connection with the Novikov conjecture in Topology and a coarse version of the Baum-Connes conjecture in Noncommutative Geometry (see \cite{NowakYu2012}, \cite{Valette2002}, \cite{FerryRanickiRosenberg1995} and \cite{Rosenberg2016}). A very efficient and powerful technique for proving the coarse geometric Novikov conjecture for a bounded geometry metric space is to coarsely embed this space into a Banach space with some strong enough convexity and smoothness properties. Indeed, G. Yu \cite{Yu2000} showed that a discrete metric space with bounded geometry that is coarsely embeddable into a Hilbert space satisfies the coarse geometric Novikov conjecture. This result was later generalized by Kasparov and Yu \cite{KasparovYu2006} who proved that the conclusion holds if the space is merely coarsely embeddable into a super-reflexive Banach space. It is worth mentioning that the coarse geometric Novikov conjecture implies Gromov's conjecture that a uniformly contractible Riemannian manifold with bounded geometry cannot have uniformly positive scalar curvature, and the zero-in-the spectrum conjecture stating that the Laplace operator acting on the space of all $L^2$-forms of a uniformly contractible Riemannian manifold has zero in its spectrum.  Consequently, whether the separable Hilbert space is the Banach space into which it is the hardest to embed, became a very natural and intriguing question. More precisely, the following problem was raised (Problem 14 in \cite{GLZ2014}, Problem 11.17 in \cite{Ostrovskiibook}).

\begin{mainp}\label{l2pb}
Does $\ell_2$ coarsely embed into every infinite dimensional Banach space?
\end{mainp}

There is rather strong evidence that $\ell_2$ is the space into which it is the most difficult to embed. It was shown in \cite{Ostrovskii2009} that every locally finite metric subset of $\ell_2$ admits a bi-Lipschitz embedding into every infinite dimensional Banach space, and in \cite{BaudierLancien2015} that every proper subset of $\ell_2$ (i.e. whose closed balls are compact) is almost Lipschitz embeddable into every infinite dimensional Banach space. Both proofs use Dvoretzky's theorem \cite{Dvoretzky1961} and the barycentric gluing technique introduced in \cite{Baudier2007}. Moreover, using Gaussian random variables one can embed $\ell_2$ linearly isometrically into $L_p[0,1]$ for all $p\in[1,\infty)$. Nowak \cite{Nowak2006} proved that $\ell_2$ coarsely embeds into $\ell_p$ for all $p\in[1,\infty)$ using Mazur maps. This technique was extended by Ostrovskii \cite{Ostrovskii2009} to embed $\ell_2$ coarsely into every Banach space with an unconditional basis and non-trivial cotype using the Mazur maps constructed by Odell and Schlumprecht in their solution to the distortion problem \cite{OdellSchlumprecht1994}. There are also Banach spaces with trivial cotype such as $\sumspace{n=1}{\infty}{\ell_\infty^n}{\ell_p}$, or $\co$ (which is a universal space for separable spaces and bi-Lipschitz embeddings \cite{Aharoni1974}) that contain coarsely $\ell_2$.

\medskip

On the other hand, useful coarse invariants, such as asymptotic dimension or finite decomposition complexity, are not well suited to study the coarse geometry of infinite dimensional Banach spaces, and the number of properties at our disposal to prevent coarse embeddability between infinite dimensional Banach spaces is rather limited. There are essentially three obstructions. The first obstruction was discovered by Johnson and Randrianarivony \cite{JohnsonRandrianarivony2006} when they showed that $\ell_q$ does not coarsely embed into $\ell_2$ when $q>2$. Their approach relies heavily upon the fact that the target space is $\ell_2$. The second obstruction is the metric cotype of Mendel and Naor \cite{MendelNaor2008}. For instance, the metric cotype argument can be used to rule out the coarse embeddability of $\ell_q$ into $\ell_p$ when $q>\max\{p,2\}$. The last one is Kalton's property $\cQ$ \cite{Kalton2007} which serves as an obstruction to coarse embeddability into reflexive (or stable) Banach spaces\footnote{A variant of Kalton's property $\cQ$ can be useful in the non-separable setting \cite{KaltonFM11}.}. Since $\ell_2$ has the smallest possible cotype allowed for a Banach space, and is both reflexive and stable, none of the above obstructions can provide a negative answer to the main problem.

\medskip

In this article the main problem is answered negatively, using the following approach. We show that there exists an infinite dimensional Banach space in which the sequence $(H_k^\omega)_{k\ge 1}$ of $k$-dimensional Hamming graphs (over a countable set) does not equi-coarsely embed. Since it is easy, and well known, that $(H_k^\omega)_{k\ge 1}$ equi-coarsely embeds into $\ell_2$, the claim follows. Our work is inspired by a concentration inequality of Kalton and Randrianarivony \cite{KaltonRandrianarivony2008}. They proved the following: Assume that  $f$ is a Lipschitz map from $H_k^\omega$ into a reflexive  Banach space $Y$, for which  there exists $p\in (1,\infty)$, such that for any $y$ in $Y$ and any weakly null sequence $\yn$ in $Y$ we have
\begin{equation}\label{E:1}
\limsup_{n\to\infty}\|y+y_n\|_Y^p\le \|y\|_Y^p+\limsup_{n\to\infty}\|y_n\|_Y^p.
\end{equation}
Then for every two vertices $\mbar,\nbar$ in some subgraph isometric to $H_k^\omega$, one has
\begin{equation}\label{E:2}
\|f(\mbar)-f(\nbar)\|_Y\le 3\lip(f)k^{1/p}.
\end{equation}
Since  $\diam(H_k^\omega)=k$, inequality \eqref{E:2} says that a Lipschitz map on $H_k^\omega$ concentrates on a ``full'' subgraph. This concentration inequality was originally used by Kalton and Randrianarivony as an obstruction to coarse Lipschitz embeddability, but also provides information on compression exponents (cf. \cite{BaudierJTA2016} and \cite{BragaSM2017}). Reflexivity is crucial to obtain inequality \eqref{E:2}. Indeed, since $c_0$ is Lipschitz universal,  inequality \eqref{E:2} cannot hold for Lipschitz maps with values in $\co$, even though $\co$ satisfies inequality $\eqref{E:1}$ for every $p\in(1,\infty)$. In order to prevent coarse embeddability one would like to have for some $C\in(0,\infty)$, which does not depend on $k$,  the stronger inequality
\begin{equation}\label{E:3}
\|f(\mbar)-f(\nbar)\|\le C\lip(f).
\end{equation}

If we could find a reflexive Banach space $Y$ satisfying inequality \eqref{E:2} for $p=\infty$ that is,
\begin{equation}\label{E:4}
\limsup_{n\to\infty}\|y+y_n\|\le \max\big\{\|y\|,\limsup_{n\to\infty}\|y_n\|\big\},
\end{equation}
then the concentration inequality \eqref{E:3} would hold for $Y$. Unfortunately, if a Banach space $Y$ satisfies inequality \eqref{E:4} then $Y$ must contain an isomorphic copy of $\co$ (see \cite{KaltonWerner1995}). Thus it cannot be reflexive and moreover it contains a bi-Lipschitz copy of every separable metric space, which precludes any concentration inequality. As we shall see, this transition phase phenomenon is not necessarily an obstruction. There are actually Banach spaces that have a ``$\co$-like'' asymptotic  smoothness property which is strong enough to obtain the concentration inequality \eqref{E:3} but weak enough to coexist with reflexivity. A prime example of such space is Tsirelson's original space. 

\medskip

Tsirelson's original space $T^*$ was the first example of a Banach space that does not contain any isomorphic copies of $\ell_p$ or $\co$ \cite{Tsirelson1974}. The validity of the concentration inequality \eqref{E:3} for $T^*$ is pivotal to prove an unexpected rigidity result which in turn implies that coarse and uniform versions of Tsirelson's theorem hold.

\subsection{Main results}
All unexplained notation can be found in Section \ref{S:2}. The main result of this article is the following rigidity result pertaining to the spreading model structure of Banach spaces coarsely embeddable into Tsirelson's original space $T^*$.

\begin{thmA}
Let $X$ be an infinite dimensional Banach space. If $X$ coarsely embeds into $T^*$, or if its unit ball $B_X$ uniformly embeds into $T^*$, then $X$ is reflexive and all its spreading models are isomorphic to $\co$.
\end{thmA}
 
Our rigidity result shows that Tsirelson's construction is actually extremely robust in the sense that nonlinear versions of Tsirelson's theorem hold. Indeed, Corollary B below is an immediate consequence of Theorem A since for $p\in(1,\infty)$ the canonical basis of $\ell_p$ generates a spreading model that is linearly isometric to $\ell_p$, while $\ell_1$ and $\co$ are not even reflexive.

\begin{corB}
$\co$ and $\ell_p$ for $p\in[1,\infty)$ (resp. $B_{\co}$ and $B_{\ell_p}$) do not coarsely embed (resp. uniformly embed) into $T^*$.
\end{corB}

We will say that a Banach space is \textit{coarsely minimal} if it coarsely embeds into every infinite dimensional Banach space. The Main Problem asks whether $\ell_2$ is coarsely minimal. In fact, Theorem A provides a much stronger negative solution to the Main Problem. Indeed, a coarsely minimal Banach space embeds into $\ell_2$ and it must have non-trivial cotype using Mendel and Naor metric cotype notion \cite{MendelNaor2008}, but it also embeds into $T^*$ and by Theorem A it must have trivial cotype, a contradiction.

\begin{corC}
There is no coarsely minimal infinite dimensional Banach space.
\end{corC} 

The proof of Theorem A relies essentially on two ingredients. The main ingredient, which is presented in Section \ref{S:4}, is a new concentration inequality for Lipschitz maps on the $k$-dimensional infinite Hamming graph and taking values into Tsirelson's original space $T^*$.

\begin{thmD}
Let $k\in \bn$ and $f\colon (\sizeeqk{\bn},d_H)\to T^*$ be a Lipschitz map. Then there exists $\bm\in[\bn]^\omega$ such that for all $\mbar,\nbar\in[\bm]^k$ one has
\begin{equation}\label{E:1.5}
\|f(\mbar)-f(\nbar)\|\le 5\lip(f).
\end{equation}
\end{thmD}

The concentration inequality above is strong enough to rule out coarse and uniform embeddings, and can already be invoked to provide a negative solution to the Main Problem without the full strength of Theorem A. 

\medskip

The other ingredient, presented in Section \ref{S:3}, uses the structure of spreading models of a Banach space to prove positive embeddability results. Such results are obtained for three different metrics: the tree metric, the Hamming metric and the symmetric difference metric. Our embeddings are based on the existence of spreading models that are isomorphic, or not isomorphic, to $\co$. 

\medskip

The proof of Theorem A, given in Section \ref{S:4}, is obtained by analyzing the tension between the concentration inequality for the Hamming graphs and the embeddability of the Hamming graphs into spaces which have spreading models that are not isomorphic to $\co$. For readers interested mostly in the geometry of Hilbert space we included in Remark \ref{R:4.6} a short proof (which avoids the machinery of spreading models and only requires Theorem D) of the $\ell_2$-case of Corollary $B$.

\medskip

The analysis of the spreading model structure of infinite dimensional Banach spaces in Section \ref{S:3}, also lead us to a metric characterization of finite dimensionality in terms of equi-coarse embeddability of the sequence of countably branching trees of finite but arbitrarily large height, denoted by $(T_k^\omega)_{k\ge 1}$.

\begin{thmE}
For a Banach space  $Y$ the following assertions are equivalent:
\begin{enumerate}
\item $Y$ is finite-dimensional.
\item $(T_k^\omega)_{k\ge1}$ does not equi-coarsely embed into $Y$.
\end{enumerate}
\end{thmE}


We mention a few related open problems in the last section.

\section{Preliminaries}\label{S:2}

\subsection{Nonlinear embeddings}
Let $(X,d_X)$ and $(Y,d_Y)$ be two metric spaces and $f:X\to Y$. One defines
$$\rho_f(t)=\inf\big\{d_Y(f(x),f(y)) : d_X(x,y)\geq t\big\},$$
and
$$\omega_f(t)={\rm sup}\{d_Y(f(x),f(y)) : d_X(x,y)\leq t\}.$$

\noindent Note that for every $x,y\in X$,
\begin{equation}\label{expcomp}
\rho_f(d_X(x,y))\le d_Y(f(x),f(y))\le\omega_f(d_X(x,y)).
\end{equation}

The moduli $\rho_f$ and $\omega_f$ will  be called the \textit{compression modulus} and the \textit{expansion modulus} of the embedding, respectively. We adopt the convention $\sup(\emptyset)=0$ and $\inf(\emptyset)=+\infty$. The map $f$ is a \textit{coarse embedding} if $\lim_{t\to \infty}\rho_f(t)=\infty$ and $\omega_f(t)<\infty$ for all $t>0$. A map $f\colon X\to Y$ is said to be a \textit{uniform embedding} if $\lim_{t\to 0}\omega_f(t)=0$ and $\rho_f(t)>0$ for all $t>0$, i.e. $f$ is an injective uniformly continuous map whose inverse is uniformly continuous.

If one is given a family  of metric spaces $(X_i)_{i\in I}$, one says that {\em $(X_i)_{i\in I}$ equi-coarsely (resp. equi-uniformly)} embeds into $Y$ if there exist non-decreasing functions $\rho, \omega\colon [0,\infty)\to[0,\infty)$ and for all $i\in I $, maps $f_i\colon X_i\to Y$ such that $\rho\le \rho_{f_i}$, $\omega_{f_i}\le \omega$, and $\lim_{t\to \infty}\rho(t)=\infty$ and $\omega(t)<\infty$ for all $t>0$ (resp. $\lim_{t\to 0}\omega(t)=0$ and $\rho(t)>0$ for all $t>0$).

\subsection{Tree, symmetric difference and Hamming metrics}\label{S:1.2}

In this section we define the metric spaces which will be needed and studied in the sequel. For any infinite subset $\bm$ of $\bn$, let $\infin{\bm}:=\{ A\subset \bm: A\text{ infinite}\}$ and $\fin{\bm}:=\{ A\subset \bm: A\text{ finite}\}$. For $k\in \bn$, we put $\sizelek{\bm}:=\{ A\subset \bm: |A|\le k\}$ and $\sizeeqk{\bm}:=\{ A\subset \bm: |A|=k\}$ where $|A|$ denotes the cardinality of the set $A$. Elements
 of  $\infin{\bm}$ and $\fin{\bm}$
 will always be listed in an increasing order, \ie if we write $\mbar=\{m_1,m_2,\ldots \}\in\infin{\bm}$, or $\mbar=\{m_1,m_2,\ldots, m_l \}\in\fin{\bm}$  we assume that $m_1<m_2<\ldots$ or $m_1<m_2<\ldots< m_l$, respectively. We define three metrics, $d_T$, $d_\symdif$  and $d_H$ on $\fin{\bn}$. The restrictions of these metrics to the sets $\sizelek{\bm}$, $\sizeeqk{\bm}$ for some $\bm \in \infin{\bn}$ and some $k\in\bn$, will still be denoted $d_T$, $d_\symdif$  and $d_H$.

\noindent{\bf The tree metric.} For $\mbar,\nbar\in\fin{\bn}$, $\mbar\not=\nbar$, let

\begin{equation}
d_T(\mbar,\nbar):=|\mbar|+|\nbar|-2|\mbar\land\nbar|,
\end{equation} where $\mbar\land\nbar:=\{m_1,m_2,\dots,m_l\}$ if $m_1=n_1, m_2=n_2,\dots,m_l=n_l$ and $m_{l+1}\neq n_{l+1}$.

We denote by $\preceq_T$ the partial order of {\em extension} on $[\bn]^{<\omega}$, {\it i.e.}
 $\mbar\preceq_T \nbar$ if $|\mbar|\le |\nbar|$ and for all $i\le |\mbar|$ one has $m_i=n_i$. One can then define a graph structure on $\fin{\bn}$, which is our set of vertices, by declaring two vertices $\mbar,\nbar$ to be adjacent if $\mbar$ is the immediate predecessor of $\nbar$ or $\nbar$ is the immediate predecessor of $\mbar$. The graph obtained is the countably branching tree of infinite height, denoted $T_\infty^\omega(\bn)$ or simply $T_\infty^\omega$ . The countably branching tree of height $k$, $[\bn]^{\le k}$, will also be denoted $T_k^\omega$. It is easy to see that the tree metric coincides with the graph metric.

\noindent{\bf The symmetric difference metric.} For $\mbar,\nbar\in\fin{\bn}$, $\mbar\not=\nbar$, let
\begin{equation}
d_\symdif(\mbar,\nbar):=|\mbar\symdif\nbar|,
\end{equation}
where $\mbar\symdif\nbar:=(\mbar\setminus\nbar)\cup(\nbar\setminus\mbar)$ is the symmetric difference.
At some point, it will be useful for us to notice that one can define a graph structure on $\sizeeqk{\bn}$ by defining an adjacency relation as follows: two vertices $\mbar,\nbar\in \sizeeqk{\bn}$ are adjacent if and only if $|\mbar\setminus\nbar|=|\nbar\setminus\mbar|=1$. This graph is the $k$-dimensional Johnson graph over a countable set, and will be denoted $J_k^\omega(\bn)$ or simply $J_k^\omega$, and its graph metric $d_J$ coincides with the metric $\frac{d_\symdif}{2}$.

\noindent{\bf The Hamming metric.} For $\mbar,\nbar\in\fin{\bn}$, let
\begin{align}\label{E:2.9}
d_H(\mbar,\nbar):=\big|\{ i\in& \{1,2,\ldots,\min(|\mbar|,|\nbar|)\}: m_i\not=n_i\}\big|\\
    &\qquad\qquad+ \max(|\mbar|,|\nbar|)-\min(|\mbar|,|\nbar|).\notag
\end{align}

When restricted to the set $\sizeeqk{\bn}$, the metric $d_H$, can be seen as the graph metric on the Hamming graph over a countable alphabet, denoted $H_k^\omega(\bn)$ or simply $H_k^\omega$, where two vertices are adjacent if they differ in exactly one coordinate.

\begin{rem}
Both graphs $H_k^\omega$ and $J_k^\omega$ are infinite versions of finite graphs that arose from the fundamental work from the mid-20th century in Coding Theory of R. W. Hamming and S. M. Johnson, respectively.
\end{rem}

 \subsection{Spreading models}\label{S:2.3}
For the convenience of the reader we briefly recall a few useful results  from the theory of spreading models initiated by Brunel and Sucheston \cite{BrunelSucheston1974} in the 70's. We shall follow the exposition in \cite{BeauzamyLapreste1984}. An application of Ramsey's Theorem yields that every bounded sequence $\xn$ in a separable Banach space admits a subsequence $\yn$ such that for all $k\ge 1$, $(a_i)_{i=1}^k\subset \br$,
$$N(a_1,a_2,\dots,a_k):=\lim_{n_1<n_2<\dots<n_k\to\infty}\|a_1y_{n_1}+a_2y_{n_2}+\dots+a_ky_{n_k}\|$$
exists, where the limit means that for all $\vep>0$, there is $r\in\bn$ such that for all $r\le n_1<n_2<\cdots<n_k$,

\begin{equation}
\Big|\|a_1y_{n_1}+a_2y_{n_2}+\dots+a_ky_{n_k}\|-N(a_1,a_2,\cdots,a_k)\Big|<\vep.
\end{equation}

It is easy to see that that if $\ei$ denotes the canonical basis of $\coo$, the vector space of sequences in $\br$ which eventually vanish,  the formula $\|a_1e_{1}+a_2e_{2}+\dots+a_ke_{k}\|_E:=N(a_1,a_2,\cdots,a_k)$ defines a semi-norm, and that  $\ei$ is a spreading sequence in the sense that for all $k\ge 1$, $(a_i)_{i=1}^k\subset \br$,
and integers $n_1<n_2<\cdots<n_k$,
\begin{equation}
\|a_1e_{n_1}+a_2e_{n_2}+\dots+a_ke_{n_k}\|_E=\|a_1e_{1}+a_2e_{2}+\dots+a_ke_{k}\|_E.
\end{equation}

If $\xn$ does not have a converging subsequence it is simple to show that the semi-norm is actually a norm. The completion of $\coo$ for the norm $\|\cdot\|_E$ is a Banach space $E$ called a \textit{ spreading model of $X$ generated by the sequence $\xn$}, and we will refer to the sequence $\ei$ as the \textit{fundamental sequence }of the spreading model. The fundamental sequence $\ei$ is not necessarily a basis of $E$ but if the generating sequence $\xn$ is normalized and weakly null then $\ei$ is a normalized $1$-suppression unconditional basis (\cite[Proposition 1, p. 24]{BeauzamyLapreste1984}).
We call a basic sequence $(z_j)_{j=1}^\infty$ in a Banach space $Z$ {\em $c$-suppression unconditional}, for some $c\ge 1$, if
 for all  $(a_i)_{i=1}^\infty \in c_{00}$  and all $A\subset \bn$
\begin{equation}\label{E:2.7}
\Big\|\sum_{i\in A}a_iz_{i}\Big\| \le  c\Big\|\sum_{i=1}^\infty a_iz_{i}\Big\|,
\end{equation}
and  in this case we denote  by $c_s$ the smallest  number $c$ satisfying \eqref{E:2.7}.
The basic sequence $(z_j)_{j=1}^\infty$ is called {\em $c$-unconditional} if for all $(a_i)_{i=1}^\infty \in c_{00}$ and all $(\vep_i)_{i=1}^\infty \subset \{-1,1\}$ we have
\begin{equation}\label{E:2.8}
\Big\|\sum_{i=1}^\infty\vep_i a_iz_{i}\Big\| \le  c\Big\|\sum_{i=1}^\infty a_iz_{i}\Big\|,
\end{equation}
and we denote in this case the smallest number $c$ satisfying \eqref{E:2.8} by $c_u$. It is easy to see
 that basic sequences are suppression unconditional  if and only if they are unconditional, and that in this case
\begin{equation}
c_s\le c_u\le 2c_s.
\end{equation}
The function $\varphi_E(k):=\|\sum_{i=1}^k e_i\|_E$ where $\ei$ is the fundamental sequence of  a spreading model $E$ of a Banach space is usually called the {\em  fundamental function of $E$}.
We deduce the following.

 \begin{prop}\label{P:2}
Let $\xn$ be a normalized weakly null sequence in a Banach space $X$. Then for every $\vep>0$ there is a normalized weakly null basic subsequence $\yn$ of $\xn$ with basis constant $(1+\vep)$, generating a spreading model $E$ with fundamental function $\varphi_E$, such that for all $k\ge1$, for all $k\le n_1< n_2< \dots < n_{k}$, and for all $(\vep_i)_{i=1}^{k}\subset\{-1,1\}$ one has
\begin{align}
\frac{1}{2(1+\vep)}\varphi_E(k)&\le  \frac1{1+\vep}  \Big\|\sum_{i=1}^{k}\vep_ie_i\Big\|\\
 &\le \Big\|\sum_{i=1}^{k}\vep_iy_{n_i}\Big\|\le
 (1+\vep)\Big\|\sum_{i=1}^{k}\vep_ie_i\Big\|\le2 (1+\vep)\varphi_E(k).\notag
\end{align}
\end{prop}

We will be particularly interested in $\co$-spreading models, where $\co$ is the space of real valued sequences converging to $0$ equipped with the sup-norm. The following proposition follows from  \cite[Lemma 1, p. 73]{BeauzamyLapreste1984} and  \cite[Lemma 4, p. 75]{BeauzamyLapreste1984}.

\begin{prop}\label{P:3}
$X$ has a spreading model isomorphic to $\co$ if and only if for all $\vep>0$, $X$ has a spreading model $E$ whose fundamental sequence $\ei$ is $(1+\vep)$-equivalent to the canonical basis of $\co$, i.e. that for all $k\ge 1$, $(a_i)_{i=1}^k\subset \br$,
\begin{equation}
\frac{1}{(1+\vep)}\sup_{1\le i\le k} |a_i|\le \Big\|\sum_{i=1}^k a_ie_{i}\Big\|_E\le (1+\vep)\sup_{1\le i\le k} |a_i|.
\end{equation}
\end{prop}
We will also need the following observation from \cite{BeauzamyLapreste1984}.

\begin{prop}\label{P:4}\cite[Proposition 3, p. 79]{BeauzamyLapreste1984} Let $\xn$ be a normalized weakly null sequence in a Banach space $X$ that generates a spreading model $E$. Then $E$ is not isomorphic to $\co$ if and only if there exists a subsequence $\yn$ of $\xn$ such that,
\begin{equation}
\lim_{k\to\infty}\ \ \inf_{n_1<\cdots<n_k}\ \ \inf_{(\vep_i)_{i=1}^k\in \{-1,1\}^k}\Big\|\sum_{i=1}^k\vep_i y_{n_i}\Big\|=+\infty.
\end{equation}
\end{prop}

\subsection{Tsirelson spaces $T$ and $T^*$}\label{S:2.4}
We recall a construction of a space  originally designed by Tsirelson \cite{Tsirelson1974}, which is nowadays usually referred to as $T^*$, while $T$ denotes its dual space and was described by Figiel and Johnson in \cite{FigielJohnson1974}. Doing so, convenient notation and terminology that will be needed later on are introduced. Tsirelson's  original construction of $T^*$ has a geometric flavor while Figiel-Johnson construction of $T$ is more analytic.

For $E,F\in[\bn]^{<\omega}$, and $n\in\bn$  we write $E<F$ if $\max(E)<\min(F)$, and $n\le E$, or $n<E$, if $n\le \min E$, or $n< \min E$, respectively. Here we set $\max(\emptyset)=0$ and $\min(\emptyset)=\infty$. We call a sequence $(E_j)_{j=1}^n\subset [\bn]^{<\omega}$ {\em admissible } if $n\le E_1<E_2<\cdots< E_n$.
For $x=\sum_{j=1}^\infty \xi_j e_j\in c_{0}$ we call $\supp(x)=\{j\in\bn: \xi_j\not=0\}$ the {\em support of $x$} and for $E\in [\bn]^{<\omega}$ we write $E(x)=\sum_{j\in E} \xi_j e_j$. We call  a finite  sequence $(x_j)_{j=1}^n$ of elements in $c_{00}$  a {\em block sequence} if  $\supp(x_{j-1})<\supp(x_j)$, for all $2\le j\le n$. We define inductively  for each $k\in\bn_0$ a norm $\|\cdot\|_k$ on $c_{00}$.  For $x=\sum_{j=1}^\infty \xi_j e_j\in c_{00}$ we put
$\|x\|_0=\max_{j\in\bn} |\xi_j|$, and assuming that $\|\cdot\|_{k-1}$ has been defined for some $k\in\bn$, we let for $x\in c_{00}$

\begin{equation*}
\|x\|_k=\max\Big(\|x\|_{k-1}, \frac12 \max\Big\{\sum_{j=1}^n \|E_j(x)\|_{k-1}: (E_j)_{j=1}^n \text{ is admissible}\Big\} \Big).
\end{equation*}

\noindent Then we put for $x\in c_{00}$

\begin{equation*}
\|x\|=\lim_{k\to \infty} \|x\|_k=\max_{k\in\bn} \|x\|_k.
\end{equation*}

\noindent $\|\cdot \|$ is then a norm on $c_{00}$ and $T$   is defined to be the completion of $c_{00}$ with respect to $\|\cdot\|$. As observed in \cite{FigielJohnson1974},  $\|\cdot\|$ satisfies the following implicit equation
 for all $x\in T$
\begin{equation*}\label{E:3.1}
 \|x\|=\max\Big(\|x\|_0, \frac12 \sup\Big\{\sum_{j=1}^n \|E_j(x)\|: (E_j)_{j=1}^n \text{ is admissible}\Big\} \Big),
\end{equation*}

\noindent and the unit vector basis is a $1$-unconditional  basis of $T$. Then it was proven in \cite{FigielJohnson1974} that $T$ does not contain a subspace isomorphic to $\ell_1$, which, together with the easy  observation that $T$ certainly does not contain a subspace isomorphic to $c_0$, yields by James' Theorem  \cite[Theorem 2]{James1950} that $T$ must be reflexive. Let us now denote the dual of $T$ by $T^*$ (the {\em original Tsirelson's space}). We denote the unit basis of $T^*$ by $(e^*_j)_{j=1}^\infty $ (which is a $1$-unconditional basis of $T^*$) and its unit ball by $B_{T^*}$. Next to the property that $T^*$ is reflexive, the following property of $T^*$ will be essential for us:

\begin{align}\label{E:2.13}
\Big\|& \sum_{j=1}^n x^*_j\Big\|\le 2, \\
 &\text{whenever $(x^*_j)_{j=1}^n\subset B_{T^*}$ is a block sequence, with $n\!\le\!\supp(x^*_1)$.}\notag
\end{align}
Indeed, assume that $(x^*_j)_{j=1}^n$ is a block sequence in $B_{T^*}$, with $n\le\supp(x^*_1)$,  and let $x\in T$, $\|x\|=1$, be such that  $\big\|\sum_{j=1}^n x^*_j\big\|=\sum_{j=1}^n x^*_j(x)$. By the $1$-unconditionality of $\ej$  in $T$ we can assume that $\supp(x)\subset \bigcup_{j=1}^n E_j$, where $E_j=\supp(x^*_j)$, for $j=1,2,\ldots, n$. Since $(E_j)_{j=1}^n$ is admissible we obtain from \eqref{E:3.1} that

\begin{equation*}
\Big\|\sum_{j=1}^n x^*_j\Big\|=\sum_{j=1}^n x^*_j(x)=\sum_{j=1}^n x^*_j\big(E_j(x)\big)\le \sum_{j=1}^n \big\|E_j(x)\|\le 2\|x\|=2,
\end{equation*}
which proves our claim \eqref{E:2.13}.

Tsirelson's space $T^*$ is not the only reflexive space satisfying inequality \eqref{E:2.13}. For every countable ordinal $\alpha$, one can define $T^*_\alpha$ (see \cite{OSZ2007}), a higher order version of $T^*$, where the admissibility condition refers to Schreier families of order $\alpha$. These Tsirelson-type spaces are all reflexive and satisfy inequality \eqref{E:2.13}, and they are incomparable in the sense that $T_\alpha^*$ does not embed isomorphically into $T_\beta^*$ whenever $\alpha\neq \beta$.

\begin{rem}
In this article, all the results that are valid for $T^*$ also hold for every reflexive Banach space satisfying equation \eqref{E:2.13}, so in particular for $T^*_\alpha$ for any countable ordinal $\alpha$.
\end{rem}

\section{A metric characterization of finite dimensionality}\label{S:3}

We start by proving some coarse embeddability results for the metric spaces introduced in Section \ref{S:1.2}. Here, we do not need that the metrics can be seen as graph metrics. In Lemma \ref{L:3.1} below we show that if $Y$ is an infinite dimensional Banach space admitting a spreading model $E$, generated by a weakly null sequence, then there always exists a $1$-Lipschitz map into $Y$, defined on any of the spaces $(\fin{\bn},d_T)$, $(\fin{\bn},d_H)$, or $(\fin{\bn},d_\symdif)$, whose compression modulus is essentially bounded from below by the fundamental function of the spreading model.

\begin{lem}\label{L:3.1} Let $(X,d_X)$ be any of the following metric spaces: $(\fin{\bn},d_\symdif)$,  $(\fin{\bn},d_T)$, or $([\bn]^k,d_H)$, for $k\in\bn$. Let $Y$ be an infinite dimensional Banach space admitting a spreading model $E$, generated by a normalized weakly null sequence, then for every $\nu>0$ there exists a map  $f_\nu:(X,d_X)\to Y$ such that for all $x,y\in X$,
\begin{equation}\label{E:3.1.1}
\frac{1}{8(1+\nu)}\varphi_E\left(d_X(x,y)\right)\le \|f_\nu(x)-f_\nu(y)\|_Y\le d_X(x,y).
\end{equation}
\end{lem}

\begin{proof}
Let $\vep>0$ such that $(1+\vep)^2\le (1+\nu)$. By Proposition \ref{P:2}  there is a weakly null normalized basic sequence $\yn$ with basis constant  not larger than $(1+\vep)$,
 and thus the bimonotonicity constant is at most $2(1+\vep)$, generating a spreading model $E$, such that for all $k\ge1$, for all $k\le n_1< n_2< \dots < n_{k}$, and for all $(\vep_i)_{i=1}^{k}\in\{-1,1\}^k$ one has
\begin{equation}\label{E:3.2}
\Big\|\sum_{i=1}^{k}\vep_iy_{n_i}\Big\|\ge \frac{1}{1+\vep}\Big\|\sum_{i=1}^{k}\vep_ie_i\Big\| \ge \frac{1}{2(1+\vep)}\varphi_E(k).
\end{equation}
We now consider the three different cases.

\noindent
{\bf The symmetric difference metric case.}
Let $f\colon (\fin{\bn},d_\symdif)\to Y$  be defined by $f(\nbar)=\ds\sum_{i=1}^k y_{n_i}$, for $\nbar=\{n_1,n_2,\ldots, n_k\}\in[\bn]^{<\omega}$. The map $f$ is clearly $1$-Lipschitz since after cancellations  it follows for $\mbar,\nbar\in[\bn]^{<\omega}$ that
$$\|f(\mbar)-f(\nbar)\|=\Big\|\sum_{i=1}^{d_\symdif(\mbar,\nbar)}\vep_iy_{q_i}\Big\|$$ for some $q_1<q_2<\dots<q_{d_\symdif(\mbar,\nbar)}$ and $(\vep_i)_{i=1}^{d_\symdif(\mbar,\nbar)}\in\{-1,1\}^{d_\symdif(\mbar,\nbar)}$. If $d=d_\symdif(\mbar,\nbar)=2r$ is even, then $q_{r+1}\ge r$, and hence  \eqref{E:3.2} and  the assumption on the bimonotonicity constant of $\yn$ yield
\begin{align*}
\|f(\mbar)-f(\nbar)\|&\ge \frac{1}{2(1+\vep)}\Big\|\sum_{i=r+1}^{d}\vep_iy_{q_i}\Big\|\ge \frac{1}{2(1+\vep)^2}\Big\|\sum_{i=r+1}^{d}\vep_ie_i\Big\|\\
                             &\ge \frac{1}{4(1+\vep)^2}\varphi_E(r)\ge \frac{1}{8(1+\vep)^2}\varphi_E(d).
\end{align*}
If $d$ is odd, the proof is similar, as will be seen while treating the next case.

\noindent
{\bf The tree metric case.}
Let $\Phi\colon [\bn]^{<\omega}\to \bn$ be a bijection.
Let $f\colon (\fin{\bn},d_T)\to Y$ be defined by $f(\nbar)=\ds\sum_{\ubar\preceq \nbar}y_{\Phi(\ubar)}$, for $\nbar\in[\bn]^{<\omega}$.
The map $f$ is clearly $1$-Lipschitz since after cancellations one has
$$\|f(\mbar)-f(\nbar)\|=\Big\|\sum_{i=1}^{d_T(\mbar,\nbar)}\vep_iy_{q_i}\Big\|,$$ for some $q_1<q_2<\dots<q_{d_T(\mbar,\nbar)}$ and $(\vep_i)_{i=1}^{d_T(\mbar,\nbar)}\in\{-1,1\}^{d_T(\mbar,\nbar)}$. Here we only detail the case when $d=d_T(\mbar,\nbar)$ is odd, say  $d=2r-1$, with $r\in \bn$. Note that $q_r\ge r$. So, again   \eqref{E:3.2} and
 the assumption on the bimonotonicity  constant of $\yn$ imply that
\begin{align*}
\|f(\mbar)-f(\nbar)\|&
\ge \frac{1}{2(1+\vep)}\Big\|\sum_{i=r}^d\vep_i y_{n_i}\Big\|
\ge \frac{1}{2(1+\vep)^2}\Big\|\sum_{i=r}^{d}\vep_ie_{i}\Big\|\\
                             &\ge \frac{1}{4(1+\vep)^2}\varphi_E(r) \ge \frac{1}{8(1+\vep)^2}\varphi_E(2r)\ge \frac{1}{8(1+\vep)^2}\varphi_E(d).
\end{align*}

\noindent
{\bf The Hamming metric case.}
Let $\Phi\colon \bn\times\bn\to \bn$ be a bijection and let $f\colon ([\bn]^k,d_H)\to Y$ be defined by $f(\nbar)=\ds\frac12\sum_{i=1}^{k}y_{\Phi(i,n_i)}$.
The map $f$ is clearly $1$-Lipschitz since after cancellations one has $$\|f(\mbar)-f(\nbar)\|=\frac12\Big\|\sum_{i=1}^{2d_H(\mbar,\nbar)}\vep_iy_{q_i}\Big\|$$ for some $q_1<q_2<\dots<q_{2d_H(\mbar,\nbar)}$ and $(\vep_i)_{i=1}^{2d_H(\mbar,\nbar)}\in\{-1,1\}^{2d_H(\mbar,\nbar)}$. Necessarily $q_{d_H(\mbar,\nbar)}\ge d_H(\mbar,\nbar)$, and hence, as in the other cases,
\begin{align*}
\|f(\mbar)-f(\nbar)\|&\ge \frac{1}{4(1+\vep)}\Big\|\sum_{i=d_H(\mbar,\nbar)+1}^{2d_H(\mbar,\nbar)}\vep_iy_{q_i}\Big\|
                             \ge \frac{1}{8(1+\vep)^2}\varphi_E\left(d_H(\mbar,\nbar)\right).
\end{align*}

\end{proof}

\begin{rem} \label{R:3.1a} We restricted in Lemma \ref{L:3.1} the Hamming metric $d_H$ to $[\bn]^k$, for $k\in\bn$, because $d_H$ is usually defined for strings of equal length. But it is not hard to show the existence of a map
$f:[\bn]^{<\omega}\to X$, satisfying condition \eqref{E:3.1.1}, if we replace the value $8$ by a larger number.
\end{rem}

Lemma \ref{L:3.1} yields a coarse embedding of $(\fin{\bn},d_T)$, $(\fin{\bn},d_\symdif)$, and an equi-coarse embedding of the sequence $([\bn]^k,d_H)_{k\in\bn}$ into every Banach space which admits at least one spreading model, that is not isomorphic to $\co$ and is generated by a weakly null sequence. Recall that Proposition \ref{P:4} in Section \ref{S:2.3} insures that the associated fundamental function is unbounded. Then, the following proposition follows simply from Lemma \ref{L:3.1}.

\begin{cor}\label{C:3.2}
The metric spaces $(\fin{\bn},d_T)$, $(\fin{\bn},d_\symdif)$, are all coarsely embeddable,
and the sequence $([\bn]^k,d_H)_{k\in\bn}$ is equi-coarsely  embeddable into a Banach space that has a spreading model $E$, generated by a normalized weakly null sequence, which is not isomorphic to $\co$.\\
More precisely, in all three cases, for any $\nu>0$ there is a coarse embedding $f_\nu$ into $Y$  such that $f_\nu$ is $1$-Lipschitz and the modulus of compression of $f_\nu$ satisfies
$\rho_{f_\nu}(t)\ge \frac{1}{8(1+\nu)}\varphi_E(t)$, for $t>0$, where $\varphi_E$ is the fundamental function of the spreading model $E$.
\end{cor}
For  the tree metric we can  deduce  the following  embeddability  result, even if all spreading models are equivalent to $c_0$.
\begin{prop}\label{P:7} If an infinite dimensional Banach space $Y$ has a spreading model isomorphic to $c_0$ then $(T_k^\omega)_{k\ge 1}$ equi-bi-Lipschitzly embeds into $Y$. \\
More precisely, for every $\nu>0$ and every $k\in\bn$ there exists a map $f_k\colon T_k^\omega\to Y$, such that for all $x,y\in T_k^\omega$,
\begin{equation}
\frac{1}{2(1+\nu)}d_T(x,y)\le \|f_k(x)-f_k(y)\|_Y\le d_T(x,y).
\end{equation}

\end{prop}

\begin{proof} By Proposition \ref{P:3} one may assume that there exists a normalized sequence $\yn$ such that for all $k\ge1$, for all $k\le n_1< n_2< \dots < n_{k}$, and for all $(a_i)_{i=1}^{k}\subset\br$ one has
\begin{equation}
\frac{1}{(1+\nu)}\sup_{1\le i\le k} |a_i|\le\Big\|\sum_{i=1}^{k}a_iy_{n_i}\Big\|\le (1+\nu)\sup_{1\le i\le k} |a_i|.
\end{equation}
Let $\Phi_k\colon \sizelek{\bn}\to \{2k,2k+1,\dots\}$ be a bijection. For $k\in\bn$
let $f_k\colon T_k^\omega\to Y$ be defined by $f_k(\nbar)=\ds\sum_{\ubar\preceq \nbar}\sum_{\sbar\preceq \ubar}y_{\Phi(\sbar)}$, for $\nbar\in[\bn]^{\le k}$.
After cancellations one has
 $$\|f_k(\mbar)-f_k(\nbar)\|=\Big\|\sum_{i=1}^{d_T(\mbar,\nbar)}\alpha_iy_{q_i}\Big\|$$
 for some $2k\le q_1<\cdots<q_{d_T(\mbar,\nbar)}$ and $(\alpha_i)_{i=1}^{d_T(\mbar,\nbar)}\subset\{1,\dots, d_T(\mbar,\nbar)\}$ satisfying $\frac12 d_T(\mbar,\nbar) \le \sup_{1\le i\le d_T(\mbar,\nbar)} |\alpha_i|\le d_T(\mbar,\nbar)$. So
\begin{align*}
\|f_k(\mbar)-f_k(\nbar)\|&\ge \frac{1}{(1+\nu)}\sup_{1\le i\le d_T(\mbar,\mbar)} |\alpha_i|\ge \frac{1}{2(1+\nu)}d_T(\mbar,\nbar)
\end{align*}
and
\begin{align*}
\|f_k(\mbar)-f_k(\nbar)\|&\le (1+\nu)\sup_{1\le i\le d_T(\mbar,\nbar)} |\alpha_i|\le (1+\nu)d_T(\mbar,\nbar),
\end{align*}which proves our claim.
\end{proof}
\begin{rem}\label{R:2}
If an infinite dimensional Banach space $Y$ has a spreading model isomorphic to $c_0$ then, using a duality argument \cite[Proposition 1, p. 80]{BeauzamyLapreste1984}, the relationship between the $\ell_1^+$-weakly null index and the Szlenk index \cite[Theorem 4.2]{AJO2005}, and the embedding in \cite[Theorem 2.6]{BKL2010}, one can show that $T_\infty^\omega$ embeds bi-Lipschitzly into $Y$.

Moreover, since it is easy to see that every countable tree isometrically embeds into $T_\infty^\omega$ we obtain that every countable tree coarsely embeds into any infinite dimensional Banach space.
\end{rem}

In the following theorem, which includes Theorem E, a purely metric characterization of finite dimensionality in terms of graph preclusion in the coarse category is given.

\begin{thm}
Let $Y$ be a Banach space. The following assertions are equivalent:

\begin{enumerate}
\item $Y$ is finite dimensional.
\item $(T_k^\omega)_{k\ge1}$ does not equi-coarsely embed into $Y$.
\item $T_\infty^\omega$ does not coarsely embed into $Y$.
\end{enumerate}
\end{thm}

\begin{proof}
We start with a self contained elementary proof of the main equivalence $(1)\iff (2)$, which avoids the use of Remark \ref{R:2}.\\
By compactness, a finite dimensional Banach space cannot contain an infinite sequence that is bounded and separated. On the other hand, if $Y$ is an infinite dimensional Banach space, then by Rosenthal's $\ell_1$-theorem, it either contains $\ell_1$ isomorphically or it has a weakly Cauchy sequence which is not norm converging. Since $\ell_1$ contains an isometric copy of $T_\infty^\omega$, we may assume that $X$ contains a weakly Cauchy sequence $\xn$ which is not norm converging. After eventually passing to a subsequence of $\xn$, still denoted $\xn$, the sequence
$$\zn:=\Big(\frac{x_{2n-1}-x_{2n}}{\|x_{2n-1}-x_{2n}\|}\Big)_{n=1}^\infty$$
is normalized and weakly null. Therefore, either $\zn$ has a subsequence that generates a spreading model isomorphic to $\co$ and we apply Proposition \ref{P:7} or $\zn$ has a subsequence that generates a spreading model not isomorphic to $\co$ and we conclude with
Corollary \ref{C:3.2}.\\
We now finish the circle of implications: $(2)  \implies (3)$ is obvious and $(3)\implies (1)$ follows from Corollary \ref{C:3.2} and Remark \ref{R:2}.
\end{proof}

\section{The concentration inequalities}\label{S:4}
All the results of this section will be valid for the spaces $[\bn]^k$, $k\in \bn$, equipped either with the Johnson metric $d_J=\frac{d_\symdif}{2}$ or the Hamming metric $d_H$. So throughout the section, $d_{\bullet}$ will denote either $d_J$ or $d_H$. We start with a general structural result on Lipschitz maps from $([\bn]^k,d_{\bullet})$ into a reflexive space with a basis. For two vectors $x$ and $y$ in a Banach space with a basis, we shall use freely the convenient notation $x\prec y$ to mean that $\supp(x)<\supp(y)$, and $k\preceq x$ when $k\le \min(\supp(x))$, where the supports are with respect to the basis considered.

\begin{prop}\label{mainprop} Let $Y$ be a reflexive Banach space with a bimonotone basis $(e_i)_{i=1}^\infty$. Then, for all $ k,r \in \bn$, $\vep>0$, $\bm\in[\bn]^\omega$, and Lipschitz maps $f\colon ([\bm]^k,d_{\bullet})\to Y$ there exist  $\bm'\in[\bn]^\omega$ and $y\in Y$ satisfying the following:

For all $\mbar\in[\bm']^k$ there exist $r\preceq y_{\mbar}^{(1)}\prec y_{\mbar}^{(2)}\prec\dots\prec y_{\mbar}^{(k)}$, all with finite supports with respect to  $\ei$, such  that:
\begin{equation}\label{E:support}
 \|y_{\mbar}^{(i)}\|\le \lip(f), \text{ for all $i\in\{1,2,\ldots, k\}$,}
\end{equation}
and
\begin{equation}\label{E:approx}
\|f(\mbar)-(y+y_{\mbar}^{(1)}+y_{\mbar}^{(2)}+\dots+y_{\mbar}^{(k)})\|<\vep.
\end{equation}

\end{prop}

\begin{proof} We will show this result by induction on $k$. The proof is based on a standard gliding hump procedure and a few  diagonal arguments that we have chosen to detail. For $n\in \bn$, we shall denote $P_n$ the basis projection on the linear span of $\{e_1,\dots,e_n\}$.

For $k=1$ assume that we are given $r \in \bn$, $\bm\in[\bn]^\omega$, a Lipschitz map $f\colon (\bm,d_{\bullet}) \to Y$, and $\vep>0$. Since $f$ is bounded and $Y$ is reflexive, by weak compactness, there exists $y\in Y$ and $\bm_0\in[\bm]^\omega$ such that $(f(m))_{m\in \bm_0}$ converges weakly to $y$. From the weak lower semi-continuity of the norm of $Y$ we deduce that  $\|f(m)-y\|\le \lip(f)$ for all $m\in\bm_0$. Since $(f(m)-y)_{m\in\bm_0}$ is weakly null, one can find $m_1\in\bm_0$ such that
$$  \|P_r(f(m)-y)\|<\frac{\vep}{2}, \text{ for all $m\in\bm_0$, with $m\ge m_1$}.$$
Then for all for all $m\ge m_1$, $m\in\bm_0$, there exists $s_m>r$ such that
$$\|(I-P_{s_m})(f(m)-y)\|<\frac{\vep}{2}.$$
We now set $\bm':=\{m\in\bm_0\colon m\ge m_1\}$ and $y_m^{(1)}:=(P_{s_m}-P_r)(f(m)-y)$, for $m\in \bm'$. For all $m\in \bm'$, $r\preceq y_m^{(1)}$, $y_m^{(1)}$ has finite support and since the basis is bimonotone, $\|y_m^{(1)}\|\le \lip(f)$. Finally, we deduce that $\|f(m)-(y+y_m^{(1)})\|<\vep$, from the triangle inequality.\\
Assume now that our statement is proved for $k\in \bn$, and let $r\in \bn$, a Lipschitz map  $f\colon ([\bm]^{k+1},d_{\bullet})\to Y$, and $\vep>0$ be given. Using that $f([\bm]^{k+1})$ is a bounded subset of the reflexive space $Y$ and a diagonal argument, we infer the existence of $\bm_0\in[\bm]^\omega$ such that for all $\mbar\in[\bm_0]^k$, $(f(\mbar,n))_{n\in\bm_0}$ converges weakly to  some $g(\mbar)\in Y$. Using again that the norm of $Y$ is weakly lower semi-continuous, we have that $\lip(g)\le \lip(f)$. This allows us to apply our induction hypothesis to the map $g\colon [\bm_0]^k\to Y$ and to find $\bm_1\in[\bm_0]^\omega$ and $y\in Y$, such that for all $\mbar\in[\bm_1]^k$ there exist $r\preceq y_{\mbar}^{(1)}\prec y_{\mbar}^{(2)}\prec\dots\prec y_{\mbar}^{(k)}$, all with finite supports, such that for all $i\in\{1,\dots,k\}$ $\|y_{\mbar}^{(i)}\|\le \lip(f)$ and
\begin{equation}
\big\|g(\mbar)-(y+y_{\mbar}^{(1)}+y_{\mbar}^{(2)}+\dots+y_{\mbar}^{(k)})\big\|<\frac{\vep}{3}.
\end{equation}
We now fix $\mbar \in [\bm_1]^k$. Note that the weak lower semi-continuity of the norm implies that for all $n \in \bm_1$, $\|f(\mbar,n)-g(\mbar)\|\le \lip(f)$. Denote $r_{\mbar}=\max(\supp(y_{\mbar}^{(k)}))$. Since  $(f(\mbar,n)-g(\mbar))_{n\in\bm_1}$ is weakly null, there exists $N_{\mbar}\in \bm_1$ such that for all $n\in \bm_1$, $n\ge N_{\mbar}$:
$$\big\|P_{r_{\mbar}}\big(f(\mbar,n)-g(\mbar)\big)\big\|<\frac{\vep}{3}.$$
Then, for all $n\in \bm_1$, $n\ge N_{\mbar}$, there exists $s_{\mbar,n}>r_{\mbar}$ so that
$$\big\|(I-P_{s_{\mbar,n}})\big(f(\mbar,n)-g(\mbar)\big)\big\|<\frac{\vep}{3}.$$
We now set $y_{(\mbar,n)}^{(k+1)}:=(P_{s_{\mbar,n}}-P_{r_{\mbar}})(f(\mbar,n)-g(\mbar))$. We have that for all $n\in \bm_1$, $n\ge N_{\mbar}$: $y_{(\mbar,n)}^{(k+1)}$ has finite support, $y_{\mbar}^k\prec y_{(\mbar,n)}^{(k+1)}$, $\|y_{(\mbar,n)}^{(k+1)}\|\le \lip(f)$ and $$\|f(\mbar,n)-(y+y_{\mbar}^{(1)}+\dots+y_{\mbar}^{(k)}+y_{(\mbar,n)}^{(k+1)})\|<\vep.$$
We conclude the proof with one last extraction argument.\\ 
If $\bm_1=\{m_1, m_2, \dots\}$, we define $\bm':=\{m_1', m_2', \dots\}$ recursively as follows. For all $i\in\bn$, $m_i':=m_{\Phi(i)}$ where $\Phi(1)=1, \Phi(2)=2, \dots, \Phi(k)=k$, and if $\Phi(i)$ has been chosen for $i\ge k$ then $\Phi(i+1)$ is picked such that $\Phi(i+1)>\Phi(i)$ and $m_{\Phi(i+1)}\ge N_{\mbar}$ for all $\mbar$ in the finite set $[\{m_1',m_2',\dots,m_i'\}]^k$. It is now easy to check that $\bm'$, $y$, and for all $\mbar\in [\bm']^{k}$ and $n\in\bm'$, $n>\max(\mbar)$, the vectors with finite support $y_{(\mbar,n)}^{(1)}:=y_{\mbar}^{(1)},\dots,y_{(\mbar,n)}^{(k)}:=y_{\mbar}^{(k)}$, $y_{(\mbar,n)}^{(k+1)}$ satisfy the induction hypothesis for $k+1$.

\end{proof}

\begin{rem}
We have assumed that the basis of $Y$ is bimonotone only for convenience. In the general case, one gets a similar result, with the only difference that $\|y_{\mbar}^{(i)}\|\le C\lip(f)$, where $C$ is  the
 bimonotonicity constant. \end{rem}

\begin{rem}
Proposition \ref{mainprop} is actually true if we only assume that $Y$ is a Banach space with a boundedly complete Schauder finite dimensional decomposition. In that case, we just have to use weak$^*$-compactness instead of weak-compactness.
\end{rem}

Our concentration result for Lipschitz maps with values in $T^*$ follows easily from Proposition \ref{mainprop}. The following theorem is Theorem C when $d_{\bullet}=d_H$.

\begin{thm}[Concentration inequality]\label{CI}
Let $k\in \bn$ and $f\colon ([\bn]^k,d_{\bullet})\to T^*$ be a Lipschitz map. Then there exists $\bm'\in[\bn]^\omega$ such that for all $\mbar,\nbar\in[\bm']^k$ one has
\begin{equation}
\|f(\mbar)-f(\nbar)\|\le 5\lip(f).
\end{equation}
\end{thm}

\begin{proof} We assume, as we may, that $\lip(f)>0$. Then it follows from  Proposition \ref{mainprop} with $Y=T^*$, $r=k$, $\bm=\bn$ and $\vep=\frac12 \lip(f)$ that there exists $\bm'\in[\bn]^\omega$, $y\in T^*$ such that for all $\mbar\in[\bm']^k$ there exist $k\preceq y_{\mbar}^{(1)}\prec y_{\mbar}^{(2)}\prec\dots\prec y_{\mbar}^{(k)}$, all with finite supports, satisfying \eqref{E:support} and \eqref{E:approx}. And hence, for all $\mbar,\nbar\in\bm'$
\begin{align*}
\|f(\mbar)-f(\nbar)\|&\le 2\vep+\|y_{\mbar}^{(1)}+y_{\mbar}^{(2)}+\dots+y_{\mbar}^{(k)}\|+\|y_{\nbar}^{(1)}+y_{\nbar}^{(2)}+\dots+y_{\nbar}^{(k)}\|\\
                             &\le 5\lip(f),
\end{align*}
where for the last inequality we use property \eqref{E:2.13} of $T^*$.
\end{proof}


The concentration inequalities stated above immediately provide an obstruction to equi-coarse embeddability of the Johnson graphs $(J_k^\omega)_{k\ge 1}$ and the Hamming graphs $(H_k^\omega)_{k\ge 1}$. A rescaling argument can be used to provide an obstruction to equi-uniform embeddability of the rescaled metrics.

\begin{cor}\label{C:11}\

(1) The sequence of Hamming graphs $(H_k^\omega)_{k\ge 1}$ (resp. $([\bn]^{k},\frac{d_H}{k})_{k\ge 1}$) do not equi-coarsely (resp. equi-uniformly) embed into $T^*$.

(2)The sequence of Johnson graphs $(J_k^\omega)_{k\ge 1}$ (resp. $([\bn]^{k},\frac{d_J}{k})_{k\ge 1}$) do not equi-coarsely (resp. equi-uniformly) embed into $T^*$.

\end{cor}

\begin{proof} We only treat the case of the Hamming graphs, the case of the Johnson graphs being  similar. Assume by contradiction that $(H_k^\omega)_{k\ge 1}$ equi-coarsely embeds into $T^*$, \ie that there exist   non-decreasing functions $\rho,\omega: [0,\infty)\to[0,\infty)$ and for each $k\in\bn$ a  function $f_k\colon ([\bn]^{k},d_H)\to T^*$, such that $\rho\le \rho_{f_k}$, $\omega_{f_k}\le \omega$, $\lim_{t\to \infty}\rho(t)=\infty$ and $\omega(t)<\infty$, for all $t>0$. Since $d_H$ is the graph distance on $H_k^\omega$, $f_k$ is actually $\omega(1)$-Lipschitz.
By Theorem \ref{CI}, for every $k\ge 1$, there exists $\bm_k\in[\bn]^\omega$ such that
for all $\mbar,\nbar\in[\bm_k]^k$ one has
\begin{equation}
\|f_k(\mbar)-f_k(\nbar)\|\le 5\omega(1).
\end{equation}
If $\mbar$ and $\nbar$ are chosen in $[\bm_k]^k$ such that $d_H(\mbar,\nbar)=k$ then it implies that for every $k\ge 1$, $\rho(k)\le 5\omega(1)$, which contradicts the properties of $\omega$ and $\rho$ for $k$ large enough.

\medskip

Assume now by contradiction that $(\sizeeqk{\bn}, \frac{d_H}{k})_{k\ge 1}$ equi-uniformly embeds into $T^*$, \ie there exist non-decreasing functions $\rho,\omega\colon [0,\infty)\to[0,\infty)$ and for each $k\in\bn$ a function $f_k\colon ([\bn]^{k},\frac{d_H}{k})\to T^*$, such that $\rho\le \rho_{f_k}$, $\omega_{f_k}\le \omega$, $\lim_{t\to 0}\omega(t)=0$ and $\rho(t)>0$, for all $t>0$. Denote  the identity map from $(\sizeeqk{\bn},d_H)$ to $(\sizeeqk{\bn}, \frac{d_H}{k})$ by  $i_k$. Then, $\lip(f_k\circ i_k)=\omega_{f_k\circ i_k}(1)\le \omega(\frac1k)$. By Theorem \ref{CI}, for all $k\ge 1$, there exist $\bm_k\in[\bn]^\omega$ such that $\|f_k(\mbar)-f_k(\nbar)\|\le 5\omega(\frac1k)$, for all $\mbar,\nbar\in [\bm_k]^k$. Therefore, considering again  $\mbar$ and $\nbar$ in $[\bm_k]^k$ such that $d_H(\mbar,\nbar)=k$, we get that for all $k\ge 1$ one has $0<\rho(1)\le5\omega(\frac1k)$, which is a contradiction for large enough $k$.
\end{proof}

\begin{rem}\label{R:4.6}
As previously mentioned, Corollary~B follows from the more involved rigidity phenomenon depicted in Theorem A. Nevertheless, for the reader only interested in the geometry of Hilbert space we include an abridged proof of the $\ell_2$-case. The proof follows simply from Corollary \ref{C:11} and, in the coarse setting, the fact that the Johnson graphs equi-coarsely embeds into $\ell_2$. Indeed,
for $k\in\bn$ consider the map

\begin{align*}
f_k\colon& (\sizeeqk{\bn},d_J)\to \ell_2,\qquad
                \nbar\mapsto \sum_{i=1}^k e_{n_i},
\end{align*}
where $(e_n)_{n=1}^\infty$ denotes the canonical basis of $\ell_2$. Then
\begin{equation}
\|f_k(\mbar)-f_k(\nbar)\|_2=\sqrt{2d_J(\mbar,\nbar)}.
\end{equation}
For our claim in the uniform category define for $k\in\bn$
\begin{align*}
g_k\colon& \Big(\sizeeqk{\bn},\frac{d_J}{k}\Big)\to B_{\ell_2},\quad
                \nbar\mapsto \frac{1}{\sqrt{k}}\sum_{i=1}^k e_{n_i},
\end{align*}
then
\begin{equation}
\|g_k(\mbar)-g_k(\nbar)\|_2=\sqrt{\frac{2d_J(\mbar,\nbar)}{k}},
\end{equation}
and we again invoke Corollary \ref{C:11} to conclude the proof of our claim.
\end{rem}

Reflexivity is usually not preserved under nonlinear embeddings, as witnessed by Ribe's example in \cite{Ribe1984}, where he showed that the two separable Banach spaces $\sumspace{n=1}{\infty}{\ell_{p_n}}{\ell_2}$ and $\sumspace{n=1}{\infty}{\ell_{p_n}}{\ell_2}\oplus\ell_1$, where $(p_n)_{n\ge 1}$ is a sequence in $(1,\infty)$ such that $\lim_{n\to \infty}p_n=1$, are uniformly homeomorphic. Using the Kalton-Randrianarivony concentration inequality, it was shown in \cite{BKL2010} that if $X$ coarse Lipschitz embeds into a reflexive Banach space that is asymptotically uniformly smooth then $X$ must be reflexive.
Actually $X$ must have the Banach-Saks property as recently observed in \cite{BragaSM2017}. Another consequence of the concentration inequality gives that reflexivity is stable under coarse embeddability into Banach spaces possessing the ``$\co$-like'' asymptotic smoothness property of Tsirelson's space $T^*$. This trade-off, in order to preserve reflexivity, between faithfulness of the embedding and smoothness properties of the target space is a rather interesting phenomenon.

\begin{prop}\label{P:12}
Let $X$ be a Banach space. If $X$ coarsely embeds, or $B_X$ uniformly embeds, into $T^*$ then $X$ is reflexive.
\end{prop}

\begin{proof}
Assume that $X$ is not reflexive. Then, by James' characterization of reflexive spaces \cite{James1972} there exists a sequence $\xn\subset B_X$ such that for all $k\ge 1$ and $\nbar\in[\bn]^{2k}$,
\begin{equation}\label{E:28}
\Big\|\sum_{i=1}^kx_{n_i}-\sum_{i=k+1}^{2k}x_{n_i}\Big\|\ge \frac{k}{2}.
\end{equation}
For every $k\in\bn$, the map $\varphi_k\colon H_k^\omega(\bn) \to X$ defined as $\varphi_k(\nbar)=\sum_{i=1}^kx_{n_i}$, for $\nbar\in[\bn]^k$, is clearly $2$-Lipschitz. Assume by contradiction that there exists a coarse embedding $f\colon X\to T^*$.
It follows again from the properties of graph distances that $f\circ \varphi_k$ is $\omega_{f\circ \varphi_k}(1)$-Lipschitz. Since $\omega_{f\circ \varphi_k}(1)\le \omega_{f}(2)$, it follows from Theorem \ref{CI} that for all $k\ge 1$, there exists $\bm_k\in[\bn]^\omega$, such that $\|f\circ\varphi_k(\mbar)-f\circ\varphi_k(\nbar)\|\le 5\omega_f(2)$, for all $\mbar,\nbar\in [\bm_k]^k$. In particular for all $k\ge 1$, if $\mbar,\nbar\in[\bm_k]^k$ are such that $m_1<m_2<\dots<m_k<n_1<n_2<\dots<n_k$, we obtain by \eqref{E:28} that $\rho_f(\frac{k}{2})\le 5\omega_f(2)$, which is impossible for large enough $k$.

For the proof in the uniform setting, we use the map $\varphi_k\colon H_k^\omega\to B_X$ defined as $\varphi_k(\nbar)=\frac{1}{k}\sum_{i=1}^kx_{n_i}$, which is clearly $\frac2k$-Lipschitz. Assume by contradiction that there exists a uniform embedding $f\colon B_X\to T^*$. Then, $\lip(f\circ \varphi_k)\le \omega_f(\frac2k)$. By Theorem \ref{CI}, for all $k\ge 1$, there exist $\bm_k\in[\bn]^\omega$ such that $\|f\circ\varphi_k(\mbar)-f\circ\varphi_k(\nbar)\|\le 5\omega_f(\frac2k)$, for all $\mbar,\nbar\in [\bm_k]^k$. In particular for all $k\ge 1$,  again if $\mbar,\nbar\in[\bm_k]^k$ are such that $m_1<m_2<\dots<m_k<n_1<n_2<\dots<n_k$, we get by \eqref{E:28} that  $0<\rho_f(\frac{1}{2})\le5\omega_f(\frac2k)$, which is again impossible if $k$ is sufficiently large.
\end{proof}
\begin{rem}
The conclusion of Proposition \ref{P:12} can be slightly strengthened as already observed in \cite{BragaSM2017}. Indeed, using \cite[Proposition 2, p. 273]{Beauzamy1979} one can show that $X$ has the Banach-Saks property.
\end{rem}

Recall that Tsirelson's space $T^*$ has the following important properties:
\begin{itemize}
\item $T^*$ is reflexive.
\item $T^*$ does not contain isomorphic copies of $\ell_p$, for any $p\in[1,\infty)$ nor of $\co$.
\item All the spreading models of $T^*$ are isomorphic to $\co$.
\end{itemize}

Theorem A shows that Tsirelson's construction is surprisingly rigid.

\begin{proof}[Proof of Theorem A] Assume first that $g:X\to T^*$ is a coarse embedding. By Corollary \ref{C:11} the sequence $(H_k^\omega)_{k\ge 1}$ does not equi-coarsely embed into $X$. Since Proposition \ref{P:12} insures that $X$ is reflexive, it follows from Rosenthal's $\ell_1$ theorem that every spreading model of $X$ can be generated by a weakly null sequence. Finally it follows from Corollary \ref{C:3.2} that every spreading model of $X$ is isomorphic to $c_0$.

\medskip
Assume now that $g:B_X\to T^*$ is a uniform embedding. Assume also by contradiction that $X$ admits a spreading model which is not isomorphic to $c_0$. By Proposition \ref{P:12}, $X$ is reflexive. Thus, it follows from Rosenthal's $\ell_1$ theorem and Proposition \ref{P:4} that there exists a normalized weakly null sequence $(x_n)_{n=1}^\infty$ in $X$ such that $\lim_{k\to \infty} \psi(k)=\infty$, where
\begin{equation}\label{E:29}
\psi(k)=\inf_{n_1<n_2<\cdots<n_k}\inf_{(\vep_i)_{i=1}^k\in \{-1,1\}}\Big\|\sum_{i=1}^k\vep_i x_{n_i}\Big\|.
\end{equation}
Define now $\varphi_k \colon ([\bn]^{k},\frac{d_J}{k})\to T^*$ by
\begin{equation}
\varphi_k(\nbar)=g\Big(\frac{1}{\psi(2k)}\sum_{i=1}^k x_{n_i}\Big),\ \  \text{for  $\nbar \in [\bn]^{k}$.}
\end{equation}
Note that $\omega_{\varphi_k}\big(\frac{1}{k})\le \omega_g(\frac{2}{\psi(2k)}\big)$. For  every $k\in\bn$  there exists $\bm_k\in[\bn]^\omega$ by  Theorem \ref{CI}, such that
\begin{equation}
\|\varphi_k(\mbar)-\varphi_k(\nbar)\|\le 5\omega_g\Big(\frac{2}{\psi(k)}\Big),\text{ whenever  $\mbar,\nbar\in[\bm_k]^k$.}
\end{equation}
But, it follows from the definition of $\psi$ that $\|\varphi_k(\mbar)-\varphi_k(\nbar)\|\ge \rho_g(1)$, whenever $d_J(\mbar,\nbar)=k$. So for all $k\in \bn$, we have $5\omega_g(\frac{2}{\psi(k)})\ge \rho_g(1)>0$. This contradicts the fact that $\lim_{k\to \infty} \psi(k)=\infty$ and $\lim_{t\to 0}\omega_g(t)=0$.
\end{proof}

\section{Final comments and open problems}
Since $T^*$ clearly has trivial cotype, it follows from Mendel-Naor metric cotype obstruction that $T^*$ is not coarsely embeddable into any Banach space that coarsely embeds into a Banach space with non-trivial type. Therefore, $T^*$ and $\ell_p$, $p\in[1,\infty)$, are coarsely incomparable in the sense that $T^*$ is not coarsely embeddable into $\ell_p$ and $\ell_p$ is not coarsely embeddable into $T^*$. To the best of our knowledge this provides the first pairs of coarsely incomparable Banach spaces.

\begin{cor}\label{C:14}
Assume that $X$ is a Banach space such that
\begin{enumerate}
\item $(H_k^\omega)_{k\ge 1}$ (or $(J_k^\omega)_{k\ge 1}$) equi-coarsely embeds into $X$,
\item $X$ coarsely embeds into a Banach space with non-trivial type.
\end{enumerate}
Then, $X$ and $T^*$ are coarsely incomparable.
\end{cor}

\medskip

Let us conclude with a few open questions.

\begin{pb}\label{P:1}
Does $\ell_2$ coarsely embed into every super-reflexive Banach space?
\end{pb}
In view of Ostrovskii's result  \cite{Ostrovskii2009}, which we referenced in the introduction, a counter-example to Problem \ref{P:1} would have to be a Banach space that is super-reflexive not containing any unconditional basic sequence.
It seems that the only known such space is Ferenczi's space \cite{Ferenczi1997}.

\begin{pb}
Does $\ell_2$ coarsely embed into any Banach space which has a spreading
model that is not equivalent to $\co$?
\end{pb}

\medskip\noindent
{\bf Acknowledgements.} This work was completed while the second named author was visiting Texas A\&M University in College Station. He wishes to thank the Mathematics Department of Texas A\&M University for its warm hospitality and the excellent working environment.

\end{document}